\documentclass[10 pt]{amsart}

\usepackage{amsmath} 
\usepackage{amsfonts}
\usepackage{amssymb}
\usepackage{amstext}
\usepackage{amsbsy}
\usepackage{amsopn}
\usepackage{amsthm}
\usepackage{amsxtra}
\usepackage{graphicx}
\usepackage{enumitem} 

\usepackage{hyperref}
\usepackage{tikz}
\usepackage{color}

\newtheorem{theorem}{Theorem}[section]
\newtheorem{lemma}[theorem]{Lemma}
\newtheorem*{lemma*}{Lemma}

\newtheorem{definition}[theorem]{Definition}
\newtheorem*{definition*}{Definition}
\newtheorem{prop}[theorem]{Proposition}

\newcommand{\Z}{\mathbb{Z}}

\newcommand{\N}{\mathbb{N}}

\newcommand{\CA}{\mathcal{A}}

\title[complexity threshold]{The complexity threshold for the emergence of Kakutani inequivalence}
\author{Van Cyr}
\address{Bucknell University, Lewisburg, PA 17837 USA}
\email{van.cyr@bucknell.edu}
\author{Aimee Johnson} 
\address{Swarthmore College, Swarthmore, PA 19081 USA}
\email{aimee@swarthmore.edu}
\author{Bryna Kra}
\address{Northwestern University, Evanston, IL 60208 USA}
\email{kra@math.northwestern.edu}
\author{Ay\c se \c Sah\. in}
\address{Wright State University, Dayton, OH 45435 USA}
\email{ayse.sahin@wright.edu}

\date{July 17, 2020}
\thanks{The authors thank the  Casa Matem\'atica Oaxaca (CMO) for hosting 
the ``Symbolic Dynamical Systems'' workshop during which this work was started.  
The third author was partially supported by NSF grant DMS-1800544.}

\begin{document}

\begin{abstract}
We show that linear complexity is the threshold for the emergence of Kakutani inequivalence for measurable systems supported on a minimal subshift.  In particular, we show that there are minimal subshifts of arbitrarily low super-linear complexity that admit both loosely Bernoulli and non-loosely Bernoulli ergodic measures and that no minimal subshift with linear complexity can admit inequivalent measures.  
\end{abstract}

\maketitle

\section{Complexity and Kakutani equivalence}
\subsection{Block complexity and constraints on the system}
The growth rate of the complexity function of a symbolic dynamical system gives rise to combinatorial invariants that allow for a finer classification of zero entropy systems and can be an obstruction for realizing certain dynamic properties.  For example, the Morse-Hedlund Theorem~\cite{MH} says that if the number of $n$ blocks in the language of a subshift grows more slowly than $n+1$ then the subshift is periodic, putting a lower bound on the growth rate for the emergence of interesting behavior.  Boshernitzan showed~\cite{Boz} that the complexity bound on a minimal subshift with linear complexity constrains the number of ergodic measures supported by the subshift and provided precise bounds related to the linear growth rate of the complexity function (two of the authors~\cite{CK} lifted the assumption of minimality when counting nonatomic measures).  Ferenczi~\cite{Fer} showed that any two ergodic measure preserving systems supported on a minimal subshift with 
linear complexity are even Kakutani equivalent and more specifically are loosely Bernoulli (see Section~\ref{sec:background} for definitions).  

For each such result, it is natural to explore the complexity threshold where the constraint is no longer present.   It was shown in~\cite{CK} that linear complexity is the threshold for Bozhernitzan's result.  In particular, given arbitrarily low superlinear complexity there exists a minimal subshift with at most that complexity that supports uncountably many invariant ergodic measures.   Here we show that linear complexity growth is also the threshold for Ferenczi's result: we show that there are minimal subshifts with arbitrarily slow superlinear complexity growth which support Kakutani inequivalent measures.

\subsection{Even Kakutani equivalence and loosely Bernoulli}
Recall that an orbit equivalence between two measurable systems $(X,T,\mu)$ and $(Y,S,\nu)$ is a bi-measurable, measure preserving map $\phi\colon X\to Y$ that maps orbits to orbits.   Ornstein's seminal result~\cite{ornstein} states that two Bernoulli systems are measure theoretically isomorphic if and only if they have equal entropy.   On the opposite end of the spectrum, Dye's theorem~\cite{dye1,dye2} states that any two measurable systems are measurably orbit equivalent if and only if they are ergodic.   
These are the two extremes of orbit equivalence, with the first
 an isomorphism preserving order on orbits and the second permuting points in an orbit without any restriction beyond measurability.  Even Kakutani equivalence lies in between: an orbit equivalence between two ergodic systems $(X,T,\mu)$ and $(Y,S,\nu)$   is an {\em even Kakutani equivalence} if there exist sets of equal positive measure $A\subset X$ and $B\subset Y$ such that the measure preserving map 
 $\phi\colon X\to Y$ is order preserving on $A$.  Namely, if and only if the induced transformations $T_A$ and $S_B$ are measurably isomorphic. 
 
Feldman, in~\cite{feldman}, was the first to use the term Kakutani equivalence to describe this equivalence relation given by inducing.  In that paper he introduced the property of loosely Bernoulli, which he showed to be an invariant for Kakutani equivalence.  He also constructed the first example of a zero entropy non-loosely Bernoulli system, which he used to build  the first example of a  K and not Bernoulli automorphism which is not  Kakutani equivalent to any Bernoulli.   He defined the loosely Bernoulli property by introducing a new metric to use in the definition of very weak Bernoulli, a key ingredient in Ornstein theory.   This new metric, called $\overline f$, weakens the $ \overline d$ metric to capture the effect of an even Kakutani equivalence on orbits. 

Feldman's work was extended by Ornstein, Rudolph, and Weiss~\cite{ORW} who developed the general equivalence theory for this metric.  They showed that the loosely Bernoulli transformations play the role for even Kakutani equivalence that the Bernoulli transformations play in the isomorphism theory.  
In particular, two loosely Bernoulli transformations are even Kakutani equivalent if and only if they have equal entropy.  This result, and the work in~\cite{ORW}, is the motivating example for the more general theory of restricted orbit equivalence developed by Rudolph in~\cite{Rud}.   
 He showed that  if an orbit equivalence satisfies certain regularity conditions then there is always a distinguished family of transformations playing the role of Bernoulli transformations for the associated equivalence relation.   Finally, we note that Katok~\cite{katok1, katok} independently defined the $\overline f$ metric and proved the equivalence theorem in the zero entropy category, using the term {\em standard} to describe the loosely Bernoulli family of transformations.

The loosely Bernoulli class of systems contains all Bernoulli transformations  but is strictly larger, even in the positive entropy category.  Here we focus our attention on the zero entropy transformations.
The simplest characterization of zero entropy loosely Bernoulli systems is that they are the family of transformations that induce rotations.  
For this reason, they are sometimes referred to as loosely Kroenecker systems.
Examples of zero entropy loosely Bernoulli systems include rotations and, 
more generally, all finite rank systems~\cite{ORW}.  
Many more examples exist, and the
study of the loosely Bernoulli property and the role of the $\overline f$ metric continues to be an active area of research; see for example~\cite{GK, KVW, GRK,KW,KL}.  

\subsection{Complexity and loosely Bernoulli}
Turning to complexity and symbolic systems, for a subshift $(X,\sigma)$, let $P(n)$ denote the block complexity of $X$, meaning 
the number of words of length $n$ that occur in any $x\in X$ (see Section~\ref{sec:background} for precise definitions).    Ferenczi~\cite[Proposition 4]{Fer} showed that if a minimal subshift has low complexity, namely $P_X(n)= O(n)$, then it has finite rank.  More generally, essentially using Ferenczi's proof, we check (see Appendix~\ref{appendix}) that the same result holds under the milder assumption that  $
\liminf_{n\to\infty} \frac{P_X(n)}{n} < \infty$. 
As finite rank transformations are loosely Bernoulli, we can rephrase Ferenczi's result in the language of Kakutani equivalence: if a minimal subshift has linear complexity, then all invariant ergodic measures on the subshift give rise to measurable systems that are even Kakutani equivalent. 

We show that linear complexity is the threshold for which this result holds.  In particular, our main result shows that there are minimal subshifts of arbitrary low super-linear complexity that admit both loosely Bernoulli and non-loosely Bernoulli ergodic measures: 
\begin{theorem}\label{thm:main}
Let $(p_n)_{n\in\N}$ be a non-decreasing sequence of integers satisfying 
\begin{equation}\label{eq:subexp}
\liminf_{n\to\infty}\frac{p_n}{n}=\infty\qquad\text{and}\qquad\limsup_{n\to\infty}\frac{\log p_n}{n}=0.
\end{equation}
Then there exists a zero entropy minimal subshift $(X,\sigma)$ satisfying 
$$ 
\liminf_{n\to\infty}\frac{P_X(n)}{p_n}=0 
$$ 
which supports two ergodic measures, $\mu$ and $\nu$, such that $(X,\sigma,\mu)$ is loosely Bernoulli while $(X,\sigma,\nu)$ is not loosely Bernoulli.
\end{theorem}
Our construction builds on Feldman's classical example from~\cite{feldman}, but requires significant modification 
in order to both find a loosely Bernoulli system supported on the subshift and to control its block complexity.

By Ferenczi's result (see the version in the Appendix), it  follows immediately that the complexity of the subshift $(X,\sigma)$ that we construct is constrained.  Namely, any subshift satisfying the conclusions of Theorem~\ref{thm:main} satisfies the superlinear growth condition  
$$ 
\liminf_{n\to\infty}\frac{P_X(n)}{n}=\infty.
$$ 

\subsection{Guide to the paper}
In Section~\ref{sec:background}, we give a short summary of the definitions and background results 
and in Section~\ref{sec:construction} we build the system $(X, \sigma)$ used to prove Theorem~\ref{thm:main}.  
Sections~\ref{sec:nlb} and~\ref{sec:lbb} are devoted to proving the existence of the measures $\mu$ and $\nu$ such that the systems $(X,\sigma,\mu)$ and  $(X,\sigma,\nu)$ are Kakutani inequivalent.  All
the results are sewn together in Section~\ref{sec:proofofmain} to prove 
Theorem~\ref{thm:main}.  In the Appendix, we review Ferenczi's result showing that linear complexity implies finite rank.

\section{Background}
\label{sec:background}

\subsection{Symbolic systems}
\label{subsec:symb}
Let $\CA$ be a finite alphabet, and denote $x\in\CA^\Z$ as $x =
(x_n)_{n\in\Z}$.
We endow  $\CA^\Z$ with the topology induced by the metric $d(x, y) = 1/2^k$ where 
$k = \inf\{|i|\colon x_i \neq y_i\}$.
The {\em left shift} $\sigma \colon \CA^\Z\to\CA^\Z$ is defined by
$(\sigma x)_n = x_{n+1}$ for all $n\in\Z$.
If $X\subset\CA^\Z$ is closed and $\sigma$-invariant, then $(X, \sigma)$
is a {\em subshift}.

The set $\CA^n$ consists of all words of length $n$ and we denote 
$w\in\CA^{n}$ by $w_0w_1 \ldots w_{n-1}  =  w_{[0,n-1]}$.  Define the {\em cylinder set
determined by $w$} to be the set
$$ [w]  =  \{x\in X\colon x_j = w_j \text{ for } j=0, \ldots, n-1\}.
$$

\begin{definition}\label{def:double-bracket}
Given $\omega\in\mathcal A^n$, we define $[[\omega]]$ to be the union of cylinder sets associated with all words of length $n$ that appear in the $2n$-length block $\omega\omega$. 
\end{definition}
If $F=\{u_1,u_2,\ldots,u_{\vert F\vert}\}$ is a collection of words $u_i$ (possibly of varying lengths), we similarly define 
$$
[F]=\bigcup_{i=1}^{\vert F\vert}[u_i]\qquad\text{and}\qquad[[F]]=\bigcup_{i=1}^{\vert F\vert}[[u_i]].
$$

Associated to the subshift $(X, \sigma)$ is the set of all shift-invariant probability measures, $\mathcal{M}(X,\sigma)$, 
defined on the Borel $\sigma$-algebra generated by the cylinder sets.  We denote 
the subset of ergodic measures by $\mathcal{M}_{e}(X,\sigma)$.  Standard results from topological dynamics 
tell us that both of these sets are nonempty.

\subsection{Complexity}
If $(X, \sigma)$ is a subshift and $n\in\N$, the {\em words $\mathcal
L_n(X)$  of length $n$}
are defined to be the collection of all $w\in\CA^n$ such that
$[w]\neq\emptyset$.  We denote the length of a word $w$ by $|w|$. 
The {\em language $\mathcal L(X)$} of the subshift $X$ is the union of all
its words:
$$\mathcal L(X) = \bigcup_{n=1}^\infty \mathcal L_n(X). $$
If $w\in\mathcal L(X)$ is a word, we say that {\em  $u\in\mathcal L(X)$
is a subword of $w$} if $w = w_1uw_2$ for some (possibly empty) words
$w_1, w_2\in\mathcal L(X)$.

For a subshift $(X, \sigma)$, the {\em word complexity $P_X\colon
\N\to\N$} is defined to be the number of words of length $n$ in the
language:
$$
P_X(n) = |\mathcal L_n(X)|.
$$
Thus $P_X(1)$ is the size of the alphabet, meaning that $P_X(1) = |\CA|$.

We say that $(X, \sigma)$ has {\em linear complexity} if
$$\liminf_{n\to\infty}P_X(n)/n < \infty.$$
One can also consider the stronger condition, $\limsup_{n\to\infty}P_X(n)/n < \infty$.  This is a distinct 
condition from the above, as there exist systems satisfying the first condition but not
satisfying the second (see~\cite{DDMP}, Example 4.1).  In this paper we restrict our attention to the $\liminf$ condition.

\subsection{Loosely Bernoulli}
\label{sec:loose}
Let $x, y\in\mathcal L(A)$ be finite words, written  
$x=x_1x_2\ldots x_m$ and $y=y_1y_2\ldots y_n$. We define a {\em match} between $x$ and $y$   
to be an order preserving bijection $\pi\colon \mathcal I^{\pi}_x\rightarrow
I^{\pi}_y$ where $\mathcal I^{\pi}_x\subset\{1,\ldots,m\}$ and 
$\mathcal I^{\pi}_{y}\subset\{1,\ldots,n\}$ with the property that $x_i=y_{\pi(i)}$.  We then 
say that the indices $i$ and $j=\pi(i)$ are {\em matched}. 
The
{\em size} of the match is defined to be
\begin{equation*}
\vert\pi\vert=\vert\mathcal I^\pi_x\vert+\vert\mathcal I^\pi_y\vert,
\end{equation*}
and the
 {\em best fit} between the two words $x$ and $y$ is
\begin{equation*}
\overline f^c(x,y)=\frac{\max\{\vert\pi\vert\colon
\text{$\pi$ is a match between $x$ and $y$}\}}{m+n}.
\end{equation*}
We can then define  the $\overline f$-distance between $x$ and $y$ to be
\begin{equation*}
\overline f(x,y)=1-  \overline f^c(x,y) = 1-\frac{\max\{\vert\pi\vert\colon \text{$\pi$ is a
match between $x$ and $y$}\}}{m+n}.
\end{equation*}
This distance measures the proportion of letters  such that, once they are deleted,  the remaining words are identical.
See for example ~\cite{feldman} or ~\cite{ORW} for more details and properties of this metric.

Let $\mathcal{B}$ be the Borel $\sigma$-algebra generated by the cylinder sets
and $\mu \in\mathcal{M}(X,\sigma)$.   
Then $(X,\mathcal{B}, \mu)$ is a Lebesgue probability space: together with $\sigma$ 
it is a measurable dynamical system.  We abbreviate this as 
$(X, \sigma, \mu)$. 

Now assume that $\mu\in \mathcal{M}_{e}(X,\sigma)$ is such that $(X, \sigma, \mu)$ is zero entropy.
In this case, we can define loosely Bernoulli as follows:
\begin{definition}\label{def:looselyBernoulli}
The zero-entropy ergodic subshift $(X, \sigma, \mu)$ is  {\em loosely Bernoulli} if for all
$\varepsilon>0$,
there exists $N>0$ such that for all $n\geq N$, there exists
$W\subset \mathcal L_n(X)$   with
\begin{itemize}
   \item $\mu(W)>1-\varepsilon$; and
   \item for any pair $\omega,\omega'\in W$, $\overline f(\omega,\omega')<\varepsilon$.
\end{itemize}
\end{definition}

\section{The Construction}\label{sec:construction}
\subsection{Feldman words}
The core symbolic structure of the subshift we build is closely related to the first example of a zero entropy non-loosely Bernoulli system that was given by Feldman ~\cite{feldman}.  We begin by describing a slight modification of his example, where changes are introduced to accommodate the additional requirements our subshift must satisfy.

Let $\{n_k\}$ be an increasing sequence of integers with $n_0\ge 2$.  We inductively define sets of words of increasing sizes.     Define the $0$th stage alphabet
\begin{eqnarray*} 
\mathcal{A}_0&:=&\{a_{0,1},a_{0,2},\dots,a_{0,n_0}\}. 
\end{eqnarray*}
For stage one of the construction we define words
\begin{eqnarray*}
\mathcal{A}_1&:=&\{a_{1,1},a_{1,2},\dots, a_{1,n_1}\}
\end{eqnarray*}
 each of length $n_0^{4n_1+3}$ by setting 
$$ 
a_{1,i}:=(a_{0,1}^{n_0^{2(i+n_1)}}a_{0,2}^{n_0^{2(i+n_1)}}\dots a_{0,n_0}^{n_0^{2(i+n_1)}})^{n_0^{2(n_1-i+1)}}\quad\text{ for  }1\leq i\leq n_1.
$$ 
For $k\ge1$, given a set of $n_k$ distinct words 
\begin{eqnarray*} 
\mathcal{A}_k&=&\{a_{k,1},a_{k,2},\dots, a_{k,n_k}\} 
\end{eqnarray*} 
each of length $\vert a_{k-1,1}\vert n_{k-1}^{4n_k+3}$, we define a set of $n_{k+1}$ new words 
\begin{eqnarray*} 
\mathcal{A}_{k+1}&=&\{a_{k+1,1},a_{k+1,2},\dots, a_{k+1,n_{k+1}}\}.
\end{eqnarray*} 
each of length $\vert a_{k,1}\vert n_{k}^{4n_{k+1}+3}$ by setting
$$ 
a_{k+1,i}=(a_{k,1}^{n_{k}^{2(i+n_{k+1})}}a_{k,2}^{n_{k}^{2(i+n_{k+1})}}\dots a_{k,n_k}^{n_{k}^{2(i+n_{k+1})}})^{n_{k}^{2(n_{k+1}-i+1)}} \quad \text{ for all }1\leq i\leq n_{k+1}. 
$$ 
In what follows we refer to the words $a_{n,k}$ as {\it Feldman words}.  Throughout our construction we introduce conditions on the growth rate of the sequence $\{n_k\}$ that guarantee that it grows rapidly enough for our subshift to have the necessary properties.

\subsection{Extended Feldman Words}
The sets of words used in our construction are extensions of the words described above.  We begin with a base-case alphabet of size $n_0+1$ to be $\mathcal A_0$ with an additional symbol $c_0$:
\begin{eqnarray*} 
\mathcal{B}_0&:=&\{a_{0,1},a_{0,2},\dots,a_{0,n_0}, c_0\}.
\end{eqnarray*}
We then define $n_1$ words of length $L_1 = 1+n_0^{4n_1+3}$ by  setting
$$ 
b_{1,i}:=(a_{0,1}^{n_0^{2(i+n_1)}}a_{0,2}^{n_0^{2(i+n_1)}}\dots a_{0,n_0}^{n_0^{2(i+n_1)}})^{(n_0)^{2(n_1-i+1)}}c_0 \quad\text{ for }1\leq i\leq n_1.
$$ 
In addition, we define a new word, also of length $L_1$, but with a different combinatorial structure: 
$$ 
c_1:=c_0^{L_1-n_0}a_{0,1}a_{0,2}\dots a_{0,n_0}.
$$ 
Finally, we denote the collection of these stage one words  by
\begin{eqnarray*}
\mathcal{B}_1&:=&\{b_{1,1},b_{1,2},\dots, b_{1,n_1},c_1\} 
\end{eqnarray*}
Note that that every letter in $\mathcal{B}_0$ appears at least once in every word in $\mathcal{B}_1$, similar to how every letter in 
$\mathcal{A}_0$ appears at least once in every Feldman word in $\mathcal{A}_1$.

We now proceed inductively.  Fix $k\geq1$ and suppose we are given the set 
\begin{eqnarray*} 
\mathcal{B}_k&=&\{b_{k,1},b_{k,2},\dots, b_{k,n_k},c_k\}
\end{eqnarray*} 
comprised of $n_k+1$ many words, all of length $L_k=(1+n_{k-1}^{4n_{k}+3}) L_{k-1}$, written in the stage-$(k-1)$ words of
 $\mathcal{B}_{k-1}$, meaning each word in $\mathcal{B}_k$ is a concatenation of words from $\mathcal{B}_{k-1}$.  
 We then  define $n_{k+1}$ words,	
all of which have length 
	\begin{equation}
	\label{eq:lengthL}
	L_{k+1}=\left(1+n_{k}^{4n_{k+1}+3}\right) L_k,
	\end{equation}
by setting
\begin{equation}\label{e:setofwords}
b_{k+1,i}=(b_{k,1}^{n_{k}^{2(i+n_{k+1})}} b_{k,2}^{n_{k}^{2(i+n_{k+1})}}\dots b_{k,{n_{k}}}^{n_{k}^{2(i+n_{k+1})}})^{n_{k}^{2(n_{k+1}-i+1)}}c_k.
\end{equation} 
We define a new word, also of length $L_{k+1}$ but with a combinatorial structure similar to the words $c_k$ constructed at previous levels,
by setting 
\begin{equation}\label{e:extraword}
c_{k+1}:=c_k^{(L_{k+1}/L_k)-n_k}b_{k,1}b_{k,2}\dots b_{k,n_{k}}.
\end{equation} 
Finally, we define the collection of $(k+1)$-words to be 
\begin{eqnarray*}
\mathcal{B}_{k+1}&:=&\{b_{k+1,1},b_{k+1,2},\dots,b_{k+1,n_{k+1}},c_{k+1}\}.
\end{eqnarray*}
Note that every element of 
$\mathcal{B}_k$ appears at least once in each element of $\mathcal{B}_{k+1}$ (and similarly with $\mathcal{A}_k$ and $\mathcal{A}_{k+1}$).
It also follows  that
\begin{equation}\label{e:LkLkplus1}
\lim_{k\to\infty}L_k = \infty \,\,\, {\rm{ and }} \,\,\,\frac{L_k}{L_{k+1}}<\frac1{n_k}.
\end{equation} 

In what follows we often distinguish between words of the type $b_{k,i}$ and the word $c_k$.   As the words $b_{k,i}$ are similar in form to the Feldman words, we refer to them as  {\it extended Feldman words}, and we denote the set of extended Feldman words at stage $k$ by 
$$\mathcal B_k^{ F} =\{\omega\in\mathcal B_k\colon \omega\neq c_k\}= {\mathcal B_k} \backslash \{c_k\}.$$

 \subsection{Constructing the subshift} \label{sbsct_subshift}
We use the extended Feldman words to construct the subshift $X$. 
We emphasize that for the construction and the properties we prove about the system, 
all of the results only rely on choosing the sequence $\{n_k\}$ with sufficiently rapid growth.  
As they become necessary, we introduce new growth conditions on this sequence, clarifying when each new 
condition is needed.  For the preliminary properties and construction of the space $X$, we only 
require that the sequence satisfies $n_0 \ge 2$ and $n_k\rightarrow\infty$. 
 
Given such a sequence, by the inductive procedure described above we obtain, for each $k\geq1$, an set $\mathcal{B}_k$ comprised of words of equal length, each of which is a concatenation of words from $\mathcal{B}_{k-1}$, and every word in $\mathcal{B}_{k-1}$ appears at least once in every word in $\mathcal{B}_k$.  
Moreover note that, by construction, $b_{k,1}$ is the leftmost subword of length $L_k$ in the word $b_{k+1,1}$ for all $k\geq0$.  Therefore we can define a one-sided infinite word $b_{\infty,1}$ as the unique word whose leftmost subword of length $L_k$ is $b_{k,1}$ for all $k\geq 1$.

We choose a new symbol, denoted $*$, that was not included in $\mathcal{B}_0$ and 
we define a $\{b_{1,1},b_{1,2},\dots, b_{1,n_1},c_1,*\}$-coloring of $\mathbb{Z}$ by coloring $\N$ with the word $b_{\infty,1}$ and coloring the set $\{\ldots,-2,-1,0\}$ with $*$; call this $\mathbb{Z}$-coloring $\alpha$.  Finally let $X$ be the space:
\begin{equation}
\label{def:X}
\left\{x\in\{b_{1,1},b_{1,2},\dots, b_{1,n_1},c_1,*\}^{\Z}\colon\forall i>0\text{, }\exists j>i \text{ such that }d(x,\sigma^j\alpha)<2^{-i}\right\}, 
\end{equation}
where $\sigma$ and $d$ are the left shift and the metric as defined in Section~\ref{subsec:symb}.
Note that $*$ does not appear in any $\Z$-coloring that can be obtained by taking larger and larger left-shifts of $\alpha$, and so the elements of $X$ are actually written in the letters $\mathcal{B}_0$.  
Thus we make no abuse of notation by referring  to $X$ as a subshift of $\mathcal{B}_0^{\Z}$.  

\begin{prop}\label{p:minimal}
The system $(X,\sigma)$, where $X$ is the space defined in~\eqref{def:X} and $\sigma\colon X\to X$ is the left shift, is minimal. 
\end{prop}
\begin{proof}
Let $u\in\mathcal{L}(X)$.  Then by the construction of $X$, $u$ occurs as a subword of $b_{\infty,1}$ and therefore there exists some $n$ such that $u$ occurs as a subword of $b_{n,1}$.  Note that every word in $\mathcal{B}_{n+1}$ contains $b_{n,1}$ as a subword (in fact every word in $\mathcal{B}_{n+1}$ contains every word in $\mathcal{B}_n$ as a subword).  We also have that for all $k>1$, $b_{n+k,1}\in \mathcal{B}_{n+k}$
can be written as a concatenation of words in $\mathcal{B}_{n+1}$.  Thus any subword of length at least $2 L_{n+1}$ in $b_{n+k,1}$ has 
$u$ as a subword.  Therefore any subword of length at least $2 L_{n+1}$ in $b_{\infty,1}$ has $u$ as a subword.  We conclude that if $v\in\mathcal{L}(X)$ is a word satisfying $|v|\geq 2 L_{n+1}$ then $u$ occurs as a subword of $v$  This means that $u$ occurs syndentically in every element of the subshift $X$ and the maximum gap between consecutive occurrences of $u$ is at most $2 L_{n+1}$.

Since $u$ was arbitrary, we conclude that every word in $\mathcal{L}(X)$ occurs syndetically in every element of $X$ and, for any fixed word in $\mathcal{L}(X)$, the maximum gap between consecutive occurrences is uniform throughout $X$ (but may depend on the word itself).  Therefore $X$ is minimal.
\end{proof}

\section{Non-loosely Bernoulli}
\label{sec:nlb}
\subsection{Overview of the existence of a non-loosely Bernoulli measure}
The goal of this section is to show 
that if the sequence $\{n_k\}$ used in constructing the space $X$ grows sufficiently rapidly, 
then there is a non-loosely Bernoulli measurable system supported on $X$.
For ease of exposition, in what follows, we assume that the sequence grows sufficiently rapidly and we defer defining the explicit growth condition to later in the section where we provide proofs of the key results.

\begin{theorem}\label{t:nlb}
If  the sequence $\{n_k\}$ used in constructing the space $X$ grows sufficiently rapidly, 
then there exists an ergodic measure $\nu\in \mathcal M_e(X)$ 
such that the system $(X,\sigma, \nu)$ is not loosely Bernoulli.
\end{theorem}

The proof follows from several propositions, which we now state, deferring their more technical proofs until after 
the proof of the theorem. 
The first proposition shows that extensions of different Feldman words from the same stage of the construction do not match well in the $\overline f$ metric:    
\begin{prop}\label{p:asbadmatch}
If the sequence $\{n_k\}$ grows sufficiently rapidly, then for all integers $r,s, k\ge0$, if $i\neq j$ then
\begin{equation}\label{e:badmatch}
\overline f(b_{k,i}^r,b_{k,j}^s)\ge \frac 58.
\end{equation}
\end{prop}

The next two results establish necessary properties of ergodic measures of extended Feldman words. 
\begin{prop}\label{p:notmanycs}
If the sequence $\{n_k\}$ grows sufficiently rapidly, then there exists a 
measure $\nu\in \mathcal M_e(X)$ such that 
$$
\lim_{k\to\infty}\nu\left([[\mathcal B_k^{F}]]\right)=1.
$$
\end{prop}

\begin{prop}\label{prop:asequal}
Given any $\xi\in\mathcal M_e(X)$ and $\varepsilon>0$, there exists $K\in\mathbb N$ such that for all $k\ge K$ and  $b_{k,m}, b_{k,j}\in\mathcal{B}_k$,
we have 
$$ | \xi ([[b_{k,m}]]) - \xi([[b_{k,j}]]) |<\varepsilon$$
\end{prop} 

With these results in hand, the proof of the theorem follows quickly: 
\begin{proof}[Proof of Theorem~\ref{t:nlb}]
We prove the result by contradiction.  
Assume that that the sequence $\{n_k\}$ grows sufficiently rapidly such that Propositions~\ref{p:asbadmatch}, \ref{p:notmanycs}, and~\ref{prop:asequal} hold, and assume the constructed system $(X,\sigma,\nu)$ is loosely Bernoulli. 
Fix  $\varepsilon>0$.   

Using Propositions~\ref{p:notmanycs} and~\ref{prop:asequal} and the fact that $\vert\mathcal B_k\vert\rightarrow\infty$, we can choose $K$ such that for all $k\ge K$ we have
$\frac1{\vert\mathcal B_k\vert}<\varepsilon$, 
$\nu\left([[\mathcal B_k^{F}]]\right)>1-\varepsilon$, 
and $ | \nu ([[b_{k,m}]]) - \nu([[b_{k,j}]]) |<\varepsilon$ for all $m,j$.

Then for any set of words $W\subset\mathcal{L}_{L_k}(X)$ with $\nu([W])>1-\varepsilon$, we have
\begin{equation}
\nu\left([W]\cap\mathcal [[B_k^{ F}]]\right)>1-2\varepsilon.
\end{equation}
Since there are more than $1/\varepsilon$ sets of the form $[[b_{k,i}]]$, each with similar measure, 
there must be distinct words $u,v\in W$ that are elements of $[[b_{k,m}]]$ and $[[b_{k,j}]]$, 
respectively, for some $m\neq j$.   
Since $\vert u\vert=\vert v\vert=L_k$, the words 
$u$ and $v$ must cover exactly half of $b^2_{k,m}$ and $b^2_{k,j}$, respectively.
 Suppose $\overline f(u,v)<\varepsilon$ and thus $\overline f^c(u,v) > 1-\varepsilon$. 
 Extending the match that realizes this value to all of $b^2_{k,m}$ and $b^2_{k,j}$
 gives that $\overline f^c(b^2_{k,m}, b^2_{k,j})\geq \frac{1}{2} (1-\varepsilon)$.  Equivalently, 
 this means that $\overline f(b^2_{k,m}, b^2_{k,j})< \frac{1}{2}(1+\varepsilon)$. 
 But for sufficiently small $\varepsilon$, this is a contradiction of Proposition~\ref{p:asbadmatch}. 
\end{proof}

The remainder of this section is devoted to the proofs of the three propositions.  

\subsection{Proof of Proposition~\ref{p:asbadmatch}:  bad $\overline f$ match of extensions of Feldman words}
We begin by proving that the Feldman words themselves do not match well in $\overline f$, and then use the fact that their extensions add only a small proportion of symbols to obtain our result. 

The first lemma is essentially Feldman's original argument in~\cite{feldman}.  We include it here for the sake of completeness.
\begin{lemma}\label{p:feldman}
If $\{n_k\}$ increases sufficiently rapidly then 
for all integers $r,s,k\ge 0$, if $i\neq j$ then 
\begin{equation}\label{e:fbarnotmatch}
\overline f(a_{k,i}^r,a_{k,j}^s)\ge\frac78.
\end{equation}
\end{lemma}
\begin{proof}
Let $\{n_k\}$ be an increasing sequence with the property that 
\begin{equation}\label{e:nkcond}
\prod_{k=0}^\infty\frac{n_k}{n_k-2}\le2\quad\text{ and } \quad\sum_{k=0}^\infty\frac{2}{n_k}\le\frac1{32}.
\end{equation}
We construct a sequence 
$\{\Gamma_k\}$ by setting 
 $\Gamma_0=0$ and for $k\ge 1$, set
\begin{equation}\label{e:gamma}
\Gamma_k=\sum_{i=0}^{k-1}\left(\prod_{j=i}^{k-1}\frac{n_j}{n_j-2}\right)^2\frac2{n_i}.
\end{equation}

We prove the lemma by showing that for all $k\geq 0$ and $i\neq j$, the match
\begin{equation}\label{e:fbarfar}
\overline f^c(a_{k,i}^r,a_{k,j}^s)\le\Gamma_k.
\end{equation}
 Clearly~\eqref{e:fbarfar} holds for $k=0$ and so 
 assume that~\eqref{e:fbarfar} holds for some $k\ge 0$.  Our goal is to show that 
\begin{equation}\label{e:endwith}
\overline f^c(a_{k+1,i}^r,a_{k+1,j}^s)\le \Gamma_{k+1}, 
\end{equation}
and then the statement follows since  $\Gamma_k\le\frac18$ for all $k\geq 0$. 

Assume that $j=i+m$, with $m\ge 1$.  
Let $\alpha_h=a_{k,h}^{n_k^{2(i+n_{k+1})}}$ denote the building blocks of the word $a_{k+1,i}$.   
With this notation, we can rewrite
\begin{align*}
a_{k+1,i}^r&=\left(\alpha_1\ldots\alpha_{n_k}\right)^{n_k^{2(n_{k+1}-i+1)}\cdot r}\\ a_{k+1,j}^s&=\left(\alpha_1^{n_k^{2m}}\ldots \alpha_{n_k}^{n_k^{2m}}\right)^{n_k^{2(n_{k+1}-j+1)}\cdot s}
\end{align*}

Consider an arbitrary match $\pi$ between these two words and take the restriction of this match to each subword $\alpha_h^{n_k^{2m}}$ of $a_{k+1,j}^s$.  
Using the restriction, we can 
partition $a^r_{k+1,i}$ into disjoint subwords that contain indices all of which are matched to a unique $\alpha_h^{n_k^{2m}}$ .  Each such subword must have the form $\beta(\alpha_1\ldots\alpha_{n_k})^t\gamma$, 
where $\beta$ and $\gamma$ are substrings from the beginning and end, respectively, of $(\alpha_1\ldots\alpha_{n_k})$.  Thus to prove~\eqref{e:endwith} it suffices to show that for all $h$
\begin{equation}\label{e:newend}
\overline f^c(\alpha_h^{n_k^{2m}},\beta(\alpha_1\ldots\alpha_{n_k})^t\gamma)\le \Gamma_{k+1}.
\end{equation}

Instead, we consider the quantity
\begin{equation}\label{e:firstexp}
\overline f^c(\alpha_h^{n_k^{2m}},(\alpha_1\ldots\alpha_{n_k})^{t+2})
\end{equation}
where we have added at most $2\vert\alpha_h\vert n_k$ symbols to $\beta(\alpha_1\dots\alpha_{n_k})^t\gamma$ and the worst fit would be if none of these additional symbols improved the match between the original pair of strings. 
Therefore, letting $\ell_{\mathcal{O}}=\vert\alpha_h^{n_k^{2m}}\vert+\vert\beta(\alpha_1\ldots\alpha_{n_k})^t\gamma\vert$ 
denote the lengths of the original two strings being matched, we have 
\begin{align*}
\overline f^c(\alpha_h^{n_k^{2m}},(\alpha_1\dots\alpha_{n_k})^{t+2})&\ge\frac{\vert\pi\vert}{\ell_{\mathcal{O}} + 2\vert\alpha_h\vert n_k}\\
&\geq \overline f^c(\alpha_h^{n_k^{2m}},\beta(\alpha_1\dots\alpha_{n_k})^t\gamma)\left(1-\frac{2\vert\alpha_h\vert n_k}{\vert\alpha_h\vert n_k^{2m}}\right)\\
&\ge \overline f^c(\alpha_h^{n_k^{2m}},\beta(\alpha_1\ldots\alpha_{n_k})^t\gamma)\left(1-\frac2{n_k}\right).
\end{align*}
Thus to prove~\eqref{e:endwith}, it suffices to show that
\begin{equation}\label{e:newgoal}
\frac{n_k}{n_k-2}\overline f^c(\alpha_h^{n_k^{2m}},(\alpha_1\ldots\alpha_{n_k})^{t+2})\le\Gamma_{k+1}.
\end{equation}

For ease of notation, define $\omega=\alpha_h^{n_k^{2m}}$ and $\omega'=(\alpha_1\ldots\alpha_{n_k})^{t+2}.$
Consider the partition of $\omega$ into disjoint subwords $\omega_{u,v}$ corresponding to contiguous subblocks that contain (but do not necessarily consist of) indices matched by $\pi$ to a symbol in the $v$-th occurrence of $\alpha_u$ in $\omega'$.  Formally we define $\omega_{u,v}$ to be the subblock of $\omega$ corresponding to the indices in the interval $\omega_{[i_*,i^*]}$, where 
$$
i_*=\min\{i\colon \pi(i)\text{ lies in the $v$-th occurrence of $\alpha_u$ in $\omega'$}\}
$$
and 
$$
i^*=\min\{i\ge i_*\colon i\in\mathcal I^\pi_{\omega}, \text{ but } \pi(i)\text{ does not lie in the $v$-th occurrence of $\alpha_u$}\}-1,
$$
recalling that the notation 
 $\mathcal I^\pi_{\omega}$ that was introduced at the beginning of Section~\ref{sec:loose}.  
Since $\pi$ is order preserving, these blocks are disjoint and contiguous.     In order to guarantee that
$\omega_{u,v}$ be a partition of $\omega$, we add any initial (respectively, final) indices in $\omega$ that are not matched to anything in $\omega'$ to $\omega_{1,1}$ (respectively, $\omega_{n_k,t+2}$).  Note that it is possible for some $\omega_{u,v}$ to be empty.  

Adopting this notation, we have:
\begin{align*}
\overline f^c(\omega,\omega')&=\frac{2\left(\text{\# of total indices in $\omega'$ that are matched by $\pi$}\right)}{\vert\omega\vert+\vert\omega'\vert}\nonumber\\
= &\frac{2}{\vert\omega\vert+\vert\omega'\vert}\sum_{u=1}^{n_k}\left(\text{\# of indices in $\omega'$ matched by $\pi$ lying in an occurrence of $\alpha_u$}\right).
\end{align*}

Note that if $u=h$, then there is the possibility that $\pi$ provided a perfect match for every possible occurrence of $\alpha_h$ in $\omega'$.   Recall that 
\begin{equation}\label{e:lengths}
\vert\omega\vert=\vert\alpha_h\vert n_k^{2m}\text{ and }\vert\omega'\vert=\vert\alpha_h\vert(t+2) n_k.
\end{equation}
So 
 \begin{align*}
\frac{2}{\vert\omega\vert+\vert\omega'\vert}&\left(\text{ \# of indices in $\omega'$ matched by $\pi$ lying in an occurrence of $\alpha_h$}\right)\\
&\le\frac{2\vert\alpha_h\vert(t+2)}{\vert\alpha_h\vert n_k^{2m}+\vert\alpha_h\vert(t+2) n_k}
\le \frac{2}{n_k}.
\end{align*}

We now turn to the matches between $\alpha_u$ and $\omega_{u,v}$ where $u\neq h$, namely the other summands:
\begin{align}
\frac{2}{\vert\omega\vert+\vert\omega'\vert}&\sum_{u=1,u\neq h}^{n_k}\sum_{v=1}^{t+2}\left(\text{\# of indices matched in $\omega'$ lying in the $v$-th occurrence of $\alpha_u$}\right)\nonumber\\
&=\frac{1}{\vert\omega\vert+\vert\omega'\vert}\sum_{u=1, u\neq h}^{n_k}\sum_{v=1}^{t+2}\overline f^c(\omega_{u,v},\alpha_u)\left(\vert\alpha_u\vert+\vert\omega_{u,v}\vert\right).\label{e:orig}
\end{align}
Recall that $\alpha_u=a_{k,u}^p$ and $\omega_{u,v}=ba_{k,h}^{p'}c$ for some $p,p'\in\mathbb N$ and where $c$ and $d$ are the end and beginning substrings of $a_{k,h}$, respectively. As before we complete each $\omega_{u,v}$ to $a_{k,h}^{p'+2}$, 
adding at most $2\vert a_{k,h}\vert$ symbols, obtaining a match between strings where our inductive hypothesis holds.  Therefore, for each $u,v$ we have
\begin{equation}\label{e:uv}
\Gamma_k\ge\overline f^c(a_{k,h}^{p'+2},\alpha_u)\ge\overline f^c(\omega_{u,v},\alpha_u)\left(1-\frac{2}{n_k}\right).
\end{equation}
Then we have that the quantity in~\eqref{e:orig} is less than or equal to 
$$
\frac{1}{\vert\omega\vert+\vert\omega'\vert} \sum_{u=1,u\neq h}^{n_k}\sum_{v=1}^{t+2}\left(\frac{n_k}{n_k-2}
\right)\Gamma_k\left(\vert\alpha_u\vert+\vert\omega_{u,v}\vert\right).
$$
Recall that $\vert\alpha_u\vert=\vert\alpha_h\vert$ for all $u$  and so we have that this last quantity is equal to
$$
\frac{1}{\vert\omega\vert+\vert\omega'\vert}\left(\frac{n_k}{n_k-2}\right)\Gamma_k\left[n_k(t+2)\vert\alpha_h\vert+\sum_{u=1,u\neq h}^{n_k}\sum_{v=1}^{t+2}\vert\omega_{u,v}\vert\right].
$$
By~\eqref{e:lengths} and the fact that the $\omega_{u,v}$ form a partition of $\omega$, this is 
$$
\le \frac{1}{\vert\omega\vert+\vert\omega'\vert}\left(\frac{n_k}{n_k-2}\right)\Gamma_k\left(\vert\omega'\vert+\vert\omega\vert\right)=\left(\frac{n_k}{n_k-2}\right)\Gamma_k.\\
$$
Putting all this together with~\eqref{e:gamma}, we see that~\eqref{e:newgoal} is satisfied:
 \begin{equation*}
 \frac{n_k}{n_k-2}\overline f^c(\alpha_h^{n_k^{2m}},(\alpha_1\cdots\alpha_{n_k})^{t+2})\le\left(\frac{n_k}{n_k-2}\right)^2\Gamma_k+\left(\frac{n_k}{n_k-2}\right)\frac{2}{n_k}\le\Gamma_{k+1}.\quad \qedhere
\end{equation*}
  \end{proof}
  
The following property of the $\overline f$ metric (see for example~\cite[Property 2.4]{GK}) is used in the proof of Proposition~\ref{p:asbadmatch}:
\begin{lemma}\label{l:gerber}
Suppose $b_1$ and $b_2$ are strings of symbols of length $n$ and $m$, respectively, from an alphabet $\mathcal A$.  If $a_1$ and $a_2$ are strings of symbols obtained by deleting at most $\lfloor\rho(n+m)\rfloor$ terms from $b_1$ and $b_2$ altogether, where $0< \rho<1$, then
\begin{equation}\label{l:remove}
\overline f(b_1,b_2)\ge \overline f(a_1,a_2)-2\rho.
\end{equation}  
\end{lemma}

\begin{proof}[Proof of Proposition~\ref{p:asbadmatch}]
Suppose the sequence $\{n_k\}$ grows sufficiently rapidly such that both~\eqref{e:nkcond} is satisfied and  
\begin{equation}\label{e:n_k}
\prod_{j=0}^\infty\frac{n_j^{4n_{j+1}+3}}{n_j^{4n_{j+1}+3}+1}>\frac78. 
\end{equation}
We then have
\begin{equation*}
\frac{\vert a_{k,i}\vert}{\vert b_{k,i}\vert}=\prod_{j=0}^{k-1}\frac{n_j^{4n_{j+1}+3}}{n_j^{4n_{j+1}+3}+1}>\frac78
\end{equation*}
and thus any two Feldman words $a_{k,j}$ and $a_{k,i}$ are obtained from the extended words $b_{k,j}$ and $b_{k,i}$ by eliminating at most $\frac18(L_k+L_k)$ symbols. 
Therefore we can apply Lemma~\ref{l:gerber} with $\rho=\frac18$ and obtain that $
\overline f(b_{k,i}^r,b_{k,j}^s)\ge\overline f(a_{k,i}^r,a_{k,j}^s)-\frac14$.  
Using the result of Lemma~\ref{p:feldman}, we can conclude that $\overline f(b_{k,i}^r,b_{k,j}^s)\ge \frac58$.  
\end{proof}

\subsection{Proofs of Propositions~\ref{p:notmanycs} and~\ref{prop:asequal}: properties of ergodic measures on $(X,\sigma)$.}
We start with the proof of Proposition~\ref{p:notmanycs}, which depends on the following lemma:
\begin{lemma}\label{lem:ergodic-large-measure}
Let $(X,\sigma,\mu)$ be a measure preserving system and let $\{b_m\}_{m=1}^{\infty}$ be a sequence of measurable sets satisfying $\mu(b_m)>1-\frac{1}{4^m}$ for all $m\geq 1$  Then there exists an ergodic measure $\nu$, supported on $X$, which satisfies $\nu(b_m)>1-\frac{1}{2^m}$ for all $m\geq 1$.
\end{lemma}
\begin{proof}
If $\mu$ is ergodic, then we are done by setting $\nu=\mu$.  Otherwise recall that $\mu$ has an ergodic decomposition, meaning there is a measurable map, $x\mapsto\mu_x$, from $X$ to the space of probability measures on $X$ satisfying with the property that $\mu_x$ is ergodic for $\mu$-almost everywhere $x\in X$, and for any measurable set $M$ we have 
$$ 
\mu(M)=\int_X\mu_x(M)d\mu(x).
$$
For each $m\geq 1$, define the measurable set 
$
a_m:=\left\{x\in X\colon\mu_x(b_m)>1-\frac{1}{2^m}\right\}. 
$ 
Then for any fixed $m$, we have 
\begin{eqnarray*} 
1-\frac{1}{4^m}&<&\mu(b_m) =\int_X\mu_x(b_m)d\mu(x) 
=\int_{a_m}\mu_x(b_m)d\mu(x)+\int_{X\setminus a_m}\mu_x(b_m)d\mu(x) \\ 
&\leq&\mu(a_m)+\left(1-\frac{1}{2^m}\right)\left(1-\mu(a_m)\right). 
\end{eqnarray*} 
Therefore $\mu(a_m)>1-\frac{1}{2^m}$ or, equivalently, $\mu(X\setminus a_m)<\frac{1}{2^m}$.  This means that 
$$
\mu\left(\bigcap_{m=1}^{\infty}a_m\right)=1-\mu\left(\bigcup_{m=1}^{\infty}(X\setminus a_m)\right) \geq1-\sum_{m=1}^{\infty}\mu(X\setminus a_m)  
$$
which gives us $\mu\left(\bigcap_{m=1}^{\infty}a_m\right)>0$.  But for $\mu$-almost every $x\in X$, the measure $\mu_x$ is ergodic and so there exists $x\in\bigcap_{m=1}^{\infty}a_m$ such that $\mu_x$ is ergodic.  Pick such an $x$ and define $\nu:=\mu_x$.  Then, since $x\in a_m$ for all $m\geq 1$, we have $\nu(b_m)=\mu_x(b_m)>1-\frac{1}{2^m}$.
\end{proof}

We are now ready to prove Proposition~\ref{p:notmanycs}, which we recall states that there is an ergodic measure on $X$ that gives large measure to the sets $[[\mathcal B_m^{F}]]$ for sufficiently large $m$. 
\begin{proof}[Proof of Proposition~\ref{p:notmanycs}]
By Lemma~\ref{lem:ergodic-large-measure}, it suffices to show that  there is a $\sigma$-invariant (but not necessarily ergodic) measure $\mu$ with the property that for all $m\geq 1$, we have 
$$
\mu\left([[\mathcal B_m^{ F}]]\right)>1-\frac1{4^m}.
$$
Suppose the sequence $\{n_k\}$ grows sufficiently rapidly such that for all $m\geq 1$, 
\begin{equation}\label{e:prodcond}
\prod_{i=m}^\infty\left(1-\frac1{n_i}\right)>1-\frac1{4^{m+1}}.
\end{equation}

Fix an arbitrary $m\ge 1$ and choose $x\in X$ such that the restriction of $x$ to its first $L_k$ 
symbols is exactly the word $b_{k,1}$ for every $k\geq 1$.  Let $\mu$ to be a weak* accumulation point of measures of the form
$$
\frac{1}{L_k}\sum_{i=0}^{L_k-1} \delta_{\sigma^ix}.
$$
Note that this means there is a subsequence $k_j$ and a value $J=J(m)$ 
such that for all $j\ge J$,
 \begin{equation}
 \label{eq:weak-star-close}
\left| \mu\left([[\mathcal B_m^{ F}]]\right) - \frac{1}{L_{k_j}}\sum_{i=0}^{L_{k_j}-1} \delta_{\sigma^ix} \left([[\mathcal B_m^{ F}]]\right) \right| < \frac{1}{4^{m+1}}.
\end{equation}

We show below that for all large $k$,
\begin{equation}
\label{eq:second-cond}
  \frac{1}{L_k}\sum_{i=0}^{L_k-1} \delta_{\sigma^ix} \left([[\mathcal B_m^{ F}]]\right)   > 1- \frac{1}{4^{m+1}},
\end{equation}
where  as usual $\delta$ denotes the Dirac measure.  Together  the two inequalities  ~\eqref{eq:second-cond} and  (\ref{eq:weak-star-close})  yield 
 that 
 $\mu\left([[\mathcal B_m^{F}]]\right) > 1-\frac{1}{4^m}$, as wanted.

It only remains to show that  ~\eqref{eq:second-cond} holds.  
By the definition of $x$,  
it suffices to count the number of $L_m$-sized subwords in $b_{k,1}$ that lie in 
$[[\mathcal B_m^{ F}]]$.
Note that $b_{k,1}$ can be thought of as the concatenation of blocks from $\mathcal{B}_{m+1}$, 
meaning the concatenation of blocks of the form $c_{m+1}$ and  $b_{m+1,j}$, $j=1,\ldots,n_{m+1}$.  
Thus, we can obtain a lower bound on the number of these $L_m$-sized subwords in $b_{k,1}$ by counting 
the number of $L_m$-sized subwords in $b_{k,1}$ that are within one of the $\mathcal{B}_{m+1}$-words  that lie in $[[\mathcal B_m^{ F}]]$.  
This quantity is in turn bounded below by the product  of the interior count $I_m$ and the multiplicity count $M_m$, 
where $I_m$ is the number of $L_m$-sized subwords within one $b_{m+1,j}$-block  that lie in $[[\mathcal B_m^{ F}]]$ 
and $M_m$ is the number of $b_{m+1,j}$-blocks in $b_{k,1}$. 
We count each of these separately.

For the interior count $I_m$, consider one $b_{m+1,j}$ block and consider all of the $L_m$-sized windows that lie in $[[\mathcal B_m^{F}]]$.  
Such a window yields a subword that is {\it{not}} in
$[[\mathcal B_m^{ F}]]$ exactly when it straddles two words of the form $b_{m,i}$ and $b_{m,i+1}$, $b_{m,n_m}$ and $b_{m,1}$, or $b_{m,n_m}$ and $c_m$.  The number of such 
subwords is at most $L_m  n_m n_m^{2(n_{m+1}+1)}$.  Since there are exactly $L_{m+1} - (L_m-1)$ subwords of length $L_m$ in any $b_{m+1,j}$   we have  
$$I_m \ge L_{m+1} - (L_m-1) - L_m n_m^{2n_{m+1}+3} \ge L_{m+1} - L_m - L_m n_m^{2n_{m+1}+3}.$$

For the multiplicity count  $M_m$, we again
view $b_{k,1}$ as the concatenation of blocks from $\mathcal{B}_{m+1}$ and count the number that are of the form
$b_{m+1,j}$ for some $j$.  We do this in steps: first view $b_{k,1}$ as the concatenation of blocks from $\mathcal{B}_{k-1}$ and note that there must be $L_k/L_{k-1}$ such blocks and all but one is of the form $b_{k-1,i}$ for some $i$.  Each of these $b_{k-1,i}$, in turn, can be thought of as the
concatenation of blocks from $\mathcal{B}_{k-2}$.  There are $L_{k-1}/L_{k-2}$ such blocks and all but one is of the form $b_{k-2,i}$ for some $i$.  
We continue in this vein, and after $k-(m+1)$ steps, we have that the number of the $\mathcal{B}_{m+1}$-blocks of the form $b_{m+1,i}$ for some $i$ is 
bounded below by
$$
\prod_{i=m+2}^k\left(\frac{L_i}{L_{i-1}}-1\right)=\prod_{i=m+2}^k\frac{L_i}{L_{i-1}}\left(1-\frac{L_{i-1}}{L_i}\right)
$$

Combining these two estimates, we have that
$$\frac{1}{L_k}\sum_{i=0}^{L_k-1} \delta_{\sigma^ix} \left([[\mathcal B_m^{F}]]\right)   \ge
\frac{1}{L_k}\left( L_{m+1} - L_m - L_m n_m^{2n_{m+1}+3} \right)  \prod_{i=m+2}^k\frac{L_i}{L_{i-1}}\left(1-\frac{L_{i-1}}{L_i}\right).$$
Since $L_k = \frac{L_k}{L_{m+1}}L_{m+1}$, we can rewrite the right hand side as
\begin{align*}
\frac{1}{L_{m+1}}& \left( L_{m+1} - L_m - L_m n_m^{2n_{m+1}+3} \right)  \frac{\prod_{i=m+2}^k\frac{L_i}{L_{i-1}}\left(1-\frac{L_{i-1}}{L_i}\right)}{\prod_{i=m+2}^k\frac{L_i}{L_{i-1}}}.
\end{align*}
Then by~\eqref{eq:lengthL} and~\eqref{e:LkLkplus1}, this last quantity is greater than or equal to 
$$
\left( 1 - \frac{1}{n_m} \right)  \prod_{i=m+2}^{k} \left( 1 - \frac{1}{n_{i-1}}\right) =\prod_{i=m}^{k-1} \left( 1 - \frac{1}{n_i}\right)\ge\prod_{i=m}^\infty\left( 1 - \frac{1}{n_i}\right),
$$
which by condition~\eqref{e:prodcond} shows that  $\frac{1}{L_k}\sum_{i=0}^{L_k-1} \delta_{\sigma^ix} \left([[\mathcal B_m^{ F}]]\right)>1-\frac1{4^{m+1}}$.
\end{proof}

We end this section by proving Proposition~\ref{prop:asequal} which assures us that all ergodic measures on $X$ for fixed $k$ give approximately the same measure to 
sets of the form $[[b_{k,m}]]$.  

\begin{proof}[Proof of Proposition~\ref{prop:asequal}]
Fix $\varepsilon>0$.  Since $n_k\rightarrow\infty$, we can choose $K\in\mathbb N$ such that for all $k\ge K$ we have $\frac2{n_k}<\frac{\varepsilon}4$.
Choose such a $k$ and consider any pair $b_{k,m}, b_{k,j}\in\mathcal{B}_k$.  By the ergodicity of $\xi$,  
we can find a point $x\in X$ such that $$
\xi  ([[b_{k,m}]]) = \lim_{n\to\infty}\frac{1}{n}\sum_{i=0}^{n-1}1_{[[b_{k,m}]]}(\sigma^ix) 
$$
and 
$$
\xi  ([[b_{k,j}]]) = \lim_{n\to\infty}\frac{1}{n}\sum_{i=0}^{n-1}1_{[[b_{k,j}]]}(\sigma^ix).
$$
We first show that it is enough to look at the frequency of the sets $[[b_{k,m}]]$ and $[[b_{k,j}]]$ in a certain subword of the point $x$.

We rewrite the difference:
\begin{eqnarray*} 
\left|  \xi([[b_{k,m}]])  -   \xi([[b_{k,j}]]) \right|   &\leq&   \left| \xi([[b_{k,m}]]) - \frac{1}{N}\sum_{i=0}^{N-1}1_{[[b_{k,m}]]}(\sigma^ix) \right|+ \\ 
&&  \left| \xi([[b_{k,j}]]) - \frac{1}{N}\sum_{i=0}^{N-1}1_{[[b_{k,j}]]}(\sigma^ix) \right| + \\
&&  \left|  \frac{1}{N}\sum_{i=0}^{N-1}1_{[[b_{k,m}]]}(\sigma^ix)  -   \frac{1}{N}\sum_{i=0}^{N-1}1_{[[b_{k,j}]]}(\sigma^ix) \right|.
\end{eqnarray*} 
Choosing $N$ large enough, we can assume that the first two terms are each bounded by $\frac\varepsilon4$.  
The last term is the difference between the number of times $\sigma^ix$ lands in the set $[[b_{k,m}]]$ as compared to $[[b_{k,j}]]$, for $i = 0$ to $N-1$.  In other words, 
this last term simply gives the difference between the number of $L_{k}$-length-subblocks of $x_{[0,N-1]}$ that are in  $[[b_{k,m}]]$ as compared to $[[b_{k,j}]]$.

Since $x$ can be written as a concatenation of words from 
$\mathcal{B}_{k+1}$, there is a subword $z$ of $x_{[0,N-1]}$ whose length is at least $N - 2 L_{k+1}$ that can be written exactly as a concatenation of 
words from $\mathcal{B}_{k+1}$.  
 We restrict our attention to this subword $z$, and let $D_N$ denote the difference between the number of $L_{k}$-length-subblocks of $z$ that are in  $[[b_{k,m}]]$ as compared to $[[b_{k,j}]]$.  Then choosing $N$ such that  $\frac{2L_{k+1}}{N} <\varepsilon/4$, we have
\begin{eqnarray*} 
 \frac{1}{N} \left| \sum_{i=0}^{N-1}1_{[[b_{k,m}]]}(\sigma^ix)  - \sum_{i=0}^{N-1}1_{[[b_{k,j}]]}(\sigma^ix) \right|  &\leq&  \varepsilon/4 + \frac{1}{N} D_N.
\end{eqnarray*} 
We thus have 
\begin{eqnarray*} 
\left|  \xi([[b_{k,m}]])  -   \xi([[b_{k,j}]]) \right|   &\leq&  3\varepsilon/4 + \frac{1}{N} D_N.
\end{eqnarray*}

We are then left with showing that $D_N/N$ is sufficiently small.  Consider the subword $z$  and divide it into disjoint subblocks of $(k+1)$-words from $\mathcal{B}_{k+1}$.  As we look at subblocks of $z$ of length $L_k$, these are either be entirely contained in one of these $(k+1)$-words or partially overlapping two adjacent $(k+1)$-words.

Let us first consider the $L_k$-length-subblocks of the second type, those overlapping two adjacent 
$(k+1)$-words in $z$.  Since the number of ways an $L_k$-length-subblock can overlap two specific $(k+1)$-words is $L_k -1$, and the number of adjacent $(k+1)$-words in $z$ is bounded by $|z|/L_{k+1} \le N/L_{k+1}$, we can bound the difference between these $L_k$-length-subblocks that are in  $[[b_{k,m}]]$ as compared to $[[b_{k,j}]]$ by $L_k  (N/L_{k+1}) = N  (L_k/L_{k+1}) \le N  (1/n_k)$.

We next consider the $L_k$-length-subblocks of the first type, the ones entirely contained in one of the $(k+1)$-words from $\mathcal{B}_{k+1}$.  If this $(k+1)$-word is an extended Feldman word (see (\ref{e:setofwords})\,)
then we see that blocks of length $L_k$ either lie within a repeated $k$-word, 
$b_{k,\ell}^{n_{k}^{2(n_{k+1}+i)}}$, or an overlap between two $k$-words, 
$b_{k,\ell}b_{k,\ell+1}$.  Note that the number of occurrences of $[[b_{k,m}]]$ and  
$[[b_{k,j}]]$ in $b_{k,\ell}^{n_{k}^{2(n_{k+1}+i)}}$, as $\ell$ ranges from 1 to $n_k$, are exactly the same.  Thus we can bound the difference by the number of $L_k$-length-subblocks that overlap a subblock of the form $b_{k,\ell}b_{k,\ell+1}$.  This is bounded by $L_k n_k n_k^{2(n_{k+1}-i+1)}$ for one $(k+1)$-word.

In the case that the $(k+1)$-word has the form of (\ref{e:extraword}), 
then the only possible occurrences of  $[[b_{k,m}]]$ and $[[b_{k,j}]]$ occur at the end, when $c_{k+1}$ cycles through the various $b_{k,\ell}$.   We can thus bound the difference between the occurrences of these sets by $L_k n_k$, which is less than the bound used above.

Thus altogether we have that 
$$  D_N \le N  (1/n_k) + L_k n_k n_k^{2(n_{k+1}-i+1)} ({\text{number of }}(k+1){\text{-words in $z$}}).$$
Using that $N\ge ({\text{number of $(k+1)$-words in $z$}}) L_{k+1}$, we have 
$$
\frac{D_N}{N} \le \frac{1}{n_k} + \frac{L_k n_k n_k^{2(n_{k+1}-i+1)}}{L_{k+1}} < \frac{1}{n_k} + \frac{1}{n_k} = \frac{2}{n_k}.
$$
It then follows that 
\begin{equation*} 
\left|  \xi([[b_{k,m}]])  -   \xi([[b_{k,j}]]) \right|   \leq  \varepsilon. \quad\qedhere
\end{equation*} 

\end{proof} 

 \section{Loosely Bernoulli}
\label{sec:lbb}
In Section~\ref{sec:nlb}, we made use of the words $b_{k,i}$ to find a measure that yielded a non-loosely Bernoulli system.  Now we make use of the words $c_k$ to find a loosely Bernoulli system.
We begin with a result that is analogous to Proposition~\ref{p:notmanycs} in that it  shows that there is an ergodic measure that gives large measure to the sets $[[c_m]]$ for sufficiently large $m$.

\begin{prop}\label{lem:notmanybs}
If the sequence $\{n_k\}$ grows sufficiently rapidly then there exists an ergodic measure $\xi$ supported on $X$ with the property that 
\begin{equation}\label{e:bigmsrc}
\lim_{m\to\infty}\xi ([[c_m]]) =1.
\end{equation}
\end{prop}
\begin{proof}
This proof is very similar to the proof of Proposition~\ref{p:notmanycs}.  Suppose $\{n_k\}$ satisfies~\eqref{e:prodcond}.  We begin by noting that, because of 
 Lemma~\ref{lem:ergodic-large-measure}, it suffices to show that  there is a $\sigma$-invariant (but not necessarily ergodic) measure $\mu$ which satisfies 
$$ 
\mu([[c_m]])>1-\frac{1}{4^m} \quad \text{ for all }m=1,2,\dots
$$ 
We find $\mu$ by choosing $y\in X$ such that the restriction of $y$ to its first $L_k$ 
symbols is exactly the word $c_k$ for every $k$.  
Just as in Proposition~\ref{p:notmanycs}, we need only show that 
\begin{equation}
\label{eq:second-cond2}
\frac{1}{L_k}\sum_{i=0}^{L_k-1} \delta_{\sigma^iy} \left([[c_m]] \right)   >  1 -\frac{1}{4^{m+1}}.
\end{equation}

Thus it suffices to count the number of $L_m$-sized subwords of $c_k$ that lie in $[[c_m]]$, 
which is bounded below by the number of $L_m$-sized subwords of $c_k$ that lie both in $[[c_m]]$ and within one of the 
$c_{m+1}$ words which makes up $c_k$.  
We bound this last quantity with the product of the interior count $I_m$ and the multiplicity count $M_m$, 
where $I_m$ is the number of $L_m$-sized subwords within one $c_{m+1}$-block  that lie in $[[c_m]]$ 
and $M_m$ is the number of $c_{m+1}$-blocks one has when $c_k$ is thought of as the concatenation of blocks from $\mathcal{B}_{m+1}$.
We count each of these separately.

For the interior count $I_m$, consider one $c_{m+1}$ block and consider all of the $L_m$-sized windows that lie in $[[c_m]]$.  
Given the structure of $c_{m+1}$ (see (\ref{e:extraword})),
every $L_m$-sized subword that lies within the $c_m^{(L_{m+1}/L_m)-n_m}$ portion of $c_{m+1}$ lies in $[[c_m]]$.  There are $L_{m+1}-L_m n_m - L_m$ such subwords and thus
$$I_m \ge L_{m+1}  - L_m (n_m+1).$$

For the multiplicity count  $M_m$, we again
view $c_k$ as the concatenation of blocks from $\mathcal{B}_{m+1}$ and count the number that are equal to 
$c_{m+1}$.  We do this in steps: first view $c_k$ as the concatenation of blocks from $\mathcal{B}_{k-1}$ and note that there must be $L_k/L_{k-1}$ such blocks and all but $n_{k-1}$ are $c_{k-1}$'s. Each of these $c_{k-1}$, in turn, can be thought of as the 
concatenation of blocks from $\mathcal{B}_{k-2}$.  There are $L_{k-1}/L_{k-2}$ such blocks and all but $n_{k-2}$ are $c_{k-2}$'s.  Continue in this vein, and after $k-(m+1)$ steps, we have that the number of the $\mathcal{B}_{m+1}$-blocks that are $c_{m+1}$ is exactly 
$$M_m = \left( \frac{L_k}{L_{k-1}} -n_{k-1} \right)  \left( \frac{L_{k-1}}{L_{k-2}} -n_{k-2} \right)  \ldots \left( \frac{L_{m+2}}{L_{m+1}} -n_{m+1} \right).$$
Combining these two estimates, we have that 
$\frac{1}{L_k}\sum_{i=0}^{L_k-1} \delta_{\sigma^iy} \left([[c_m]]\right)$
is bounded below by
$$\frac{1}{L_k} \left( L_{m+1} - L_m(n_m+1) \right)  \prod_{i=m+2}^k \left(\frac{L_i}{L_{i-1}} - n_{i-1}\right).$$
Since $L_k = \frac{L_k}{L_{m+1}} L_{m+1}$, we can write this as 
$$\frac{1}{L_{m+1}} \left( L_{m+1} - L_m(n_m+1) \right)  
\frac{\prod_{i=m+2}^k \left(\frac{L_i}{L_{i-1}} - n_{i-1}\right)} {\prod_{i=m+2}^k \frac{L_i}{L_{i-1}} }$$
$$\frac{1}{L_{m+1}} \left( L_{m+1} - L_m(n_m+1) \right)  \prod_{i=m+2}^k  \left( 1 - n_{i-1} \frac{L_{i-1}}{L_{i}}\right).$$

Since $L_{i} = L_{i-1} (1+n_{i-1}^{4n_{i}+3})$, we have ${L_{i-1}}/{L_{i}} = 1/{(1+{n_{i-1}^{4n_{i}+3}})}$ and thus $(n_{i-1}L_{i-1})/L_{i} \le {1}/{n_{i-1}}$.  Similarly, $(n_m+1){L_m}/{L_{m+1}}\le {(n_m+1)}/{n_m^{4n_{m+1}+3}} \le {1}/{n_m}$.

We thus have
$$\frac{1}{L_k}\sum_{i=0}^{L_k-1} \delta_{\sigma^iy} \left([[c_m]]\right) \ge 
\left(1-\frac{1}{n_m}\right) \prod_{i=m+2}^{k} \left( 1 - \frac{1}{n_{i-1}}\right) .$$

It then follows from~\eqref{e:prodcond} that 
$\frac{1}{L_k}\sum_{i=0}^{L_k-1} \delta_{\sigma^iy} \left([[c_m]]\right)\ge 1-\frac{1}{4^{m+1}}$. 
\end{proof}

\begin{theorem}\label{t:lb}
If the sequence $\{n_k\}$ grows sufficiently rapidly 
then $X$ carries a measure $\xi$ such that $(X,\sigma,\xi)$ is loosely Bernoulli.
\end{theorem}

\begin{proof}
Let $\{n_k\}$ be a sequence that satisfies~\eqref{e:prodcond}.  
We apply Proposition~\ref{lem:notmanybs} to obtain a measure $\xi$ satisfying~\eqref{e:bigmsrc}.
 
Let $K$ be large enough such that both $\frac{1}{n_{K-1}} < \varepsilon$ and $\xi([[c_K]]) > 1-\varepsilon$.  Take $k\ge K$.  Let $x, y\in\mathcal{L}_{L_k}(X)$ be two words that occur as subwords of $c_kc_k$, so each is a word of length  $L_k$ which looks like the end of a $c_k$ concatenated with the beginning of a $c_k$.  

Recall that $c_k$ is a concatenation of many copies of $c_{k-1}$ followed by the block $b_{k-1,1}, \ldots , b_{k-1,n_{k-1}}$.
Therefore, by eliminating at most $L_{k-1}(n_{k-1}+2)$ indices from both $x$ and $y$, we can remove any indices corresponding to the extended Feldman words and any partial copies of $c_{k-1}$.  What is left  now are words of the form $c^r_{k-1}$ and $c^s_{k-1}$, where $r$ and $s$ could differ as different choices of $x$ and $y$ might necessitate removal of a different number of indices. 
However, $\frac{L_k}{L_{k-1}} - n_{k-1}-2 \le r, s \le \frac{L_k}{L_{k-1}} - n_{k-1}$, $\rm{i.e.}$ $|r-s|\le 2$.  Thus by throwing out at most another $2L_{k-1}$ indices, we obtain identical strings.

This means that 
$$\overline{f}(x,y) \le\frac{1}{2L_k} \left( L_{k-1}(n_{k-1}+2)+2L_{k-1}\right) = \frac{L_{k-1}(n_{k-1}+2)}{L_k} <  \frac{1}{n_{k-1}} < \varepsilon .$$
Taking $W$ and $N$ in  Definition~\ref{def:looselyBernoulli} to be the set $[[c_K]]$ and our choice of $K$ respectively yields the result.
\end{proof}

\section{Proof of Theorem~\ref{thm:main}}
\label{sec:proofofmain}
We are now ready to prove our main result, showing that the system $(X, \sigma)$ we constructed carries 
both a loosely Bernoulli measure and a non-loosely Bernoulli measure. 
\begin{proof}[Proof of Theorem~\ref{thm:main}]
We fix a non-decreasing sequence $\{p_n\}_{n\in\mathbb N}$ satisfying 
$$\liminf_{n\to\infty}\frac{p_n}{n}=\infty\quad\text{and}\quad\limsup_{n\to\infty}\frac{\log p_n}{n}=0.$$
We choose a sequence $\{n_k\}_{k\geq 0}$ satisfying all of the growth conditions needed to apply our arguments.  
More precisely, we choose this sequence such that  $n_0\ge 2$ and $n_k\rightarrow\infty$, and the 
growth conditions corresponding to~\eqref{e:nkcond}, \eqref{e:n_k}, and~\eqref{e:prodcond} are all satisfied, 
meaning that:
$$\prod_{k=0}^\infty\frac{n_k}{n_k-2}\le2\quad ; \quad \sum_{k=0}^\infty\frac{2}{n_k}\le\frac1{32} \quad ; \quad 
\prod_{k=0}^\infty\frac{n_k^{4n_{k+1}+3}}{n_k^{4n_{k+1}+3}+1}>\frac78 ; 
$$
$$
\prod_{k=m}^\infty\left(1-\frac1{n_k}\right)>1-\frac1{4^{m+1}}.
$$
We remark that while it is possible to simplify the assumptions on the growth, 
 as some of these conditions imply others, 
for clarity in how they are used, we keep each of them. As these are all growth conditions on the sequence, they are clearly compatible.   For example, the sequence $n_k=4^{k+4}$ suffices. 

We inductively modify the sequence $\{n_k\}_{k\geq 0}$, again only possibly increasing the growth rate.   Given $n_k$, choose
$n_{k+1} > n_k$ such that, in addition to the growth conditions already satisfied (meaning conditions~\eqref{e:nkcond}, \eqref{e:n_k}, and~\eqref{e:prodcond}), we also have that if $m=L_kn_k^{2n_{k+1}+2}$, then 
\begin{equation}
\label{eq:choice-of-m}
p_m > k(6+3n_k)m.
\end{equation}

Now construct a subshift $X$ as described in Section  \ref{sbsct_subshift}.  By Proposition~\ref{p:minimal}, Theorem~\ref{t:nlb}, and Theorem~\ref{t:lb} this subshift $X$ is minimal, and there exist $\nu,\xi\in\mathcal M_e(X)$ such that $(X,\sigma, \nu)$ is not loosely Bernoulli and $(X,\sigma, \xi)$ is loosely Bernoulli.

We now check that $X$ satisfies the complexity condition 
\begin{equation}\label{e:liminf}
\liminf_{n\to\infty}\frac{P_X(n)}{p_n}=0. 
\end{equation}
Fix $k$ and consider the words of length $m \, (=L_kn_k^{2n_{k+1}+2})$.
Since any point $x\in X$ can be written as a  bi-infinite 
concatenation of elements of $\mathcal{B}_{k+1}$, we can count the number of words of length $m$ by counting the number found  entirely within 
an element in $\mathcal{B}_{k+1}$ or overlapping two concatenated elements of $\mathcal{B}_{k+1}$. 

To make this count, recall the forms of the words $b_{k+1,i}$ and $c_{k+1}$ are given in~\eqref{e:setofwords} and~\eqref{e:extraword}.
The word  $b_{k+1,1}$ has each $b_{k,j}$ repeated exactly $n_k^{2(1+n_{k+1})}$ times, meaning that this portion of $b_{k+1,1}$ has length 
$L_k n_k^{2(1+n_{k+1})} = m$.  The other $b_{k+1,i}$ have even longer lengths of repeated $k$-words.

We first count the number of words of length $m$ that are subwords of some $b_{k+1,i}$.  Note that for each choice of 
$i$, one of following occurs:
\begin{enumerate}[label=(\roman*)]
  \item The word of length $m$ is a subword of a repeated $k$-word $b_{k,j}^{n_{k}^{2(i+n_{k+1})}}$ for $j=1,\ldots,n_{k}$.  The repetition of $b_{k,j}$ means the number of distinct words of length $m$  of this type for a specific $j$ is just the length of $b_{k,j}$ and thus altogether we have $L_k n_k$ number of such words of length $m$.
  \item The word of length $m$ overlaps two consecutive repeated $k$-words and thus is a subword of 
$b_{k,j}^{n_{k}^{2(i+n_{k+1})}}b_{k,j+1}^{n_{k}^{2(i+n_{k+1})}}$.  For a specific $j$, we can count the number of distinct words of length $m$ by counting the number of locations within the word at which it switches from a $b_{k,j}$  to a $b_{k,j+1}$: this number is $m$.  Since $j$ ranges from 1 to $n_k$, this yields a total of $mn_k$ distinct words.
  \item The word of length $m$ lies towards the end of $b_{k+1,i}$ and thus, for all $i$, is a subword of 
  $b_{k,{n_{k}}}^{n_{k}^{2(i+n_{k+1})}} c_k$.  The number of distinct words of length $m$ can be counted by simply noting the number of places at which $c_k$ can begin in the word, which is $L_k$.
\end{enumerate}
To count the number of words of length $m$ that are subwords of $c_{k+1}$, note that one of two situations occur:
\begin{enumerate}[label=(\roman*)]
   \item The word of length $m$ is a subword of the repeated $c_k$'s.  There are $L_k$ distinct such subwords.
   \item The word of length $m$ lies towards the end of $c_{k+1}$ and thus is a subword of $c_k \ldots c_k b_{k,1}b_{k,2}\dots b_{k,n_{k}}$.   We can count this by counting the number of locations within the word at which it switches from the $c_k$ to $b_{k,1}$, which is $n_k L_k$.
\end{enumerate}
Thus the number of words of length $m$ that are subwords of some element $\mathcal{B}_{k+1}$ is 
$2L_kn_k + mn_k +  2L_k   $.

We next count the number of words of length $m$ that overlap two concatenated elements of $\mathcal{B}_{k+1}$. There are four possible combinations for such concatenations:

$\begin{array}{llll}
{\rm{(i)}} \,\, b_{k+1,i}b_{k+1,j} & {\rm{(ii)}} \,\, c_{k+1}b_{k+1,j} & {\rm{(iii)}} \,\, b_{k+1,j} c_{k+1} & {\rm{(iv)}} \,\, c_{k+1} c_{k+1}
\end{array}$

To count the number of combinations for the first situation,  
we note that for all choices of $i$ and $j$, the word of length $m$ is a subword of 
   $b_{k,n_k} \ldots b_{k,n_k} c_k  b_{k,1} \ldots. b_{k,1}$.  To count the number of distinct words of length $m$ by noting that there are exactly $m$ locations where we can see the beginning of $c_k$ within it, which leads to a count of $m$ distinct words.
The three other possible combinations can be counted in a similar manner, with 
each yielding at most $m$ distinct words.  Thus we obtain a total of 4$m$ words of this type.

Combining these counts and using that $L_k\le m$, we have that $P_X(m) \le m(3n_k+6)$.  But by~\eqref{eq:choice-of-m}, we have that 
$mk(3n_k+6) < p_m$.  Therefore $P_X(m)\le p_m/k$.  But recall that $m=L_kn_k^{2n_{k+1}+2}$  so this shows that there is a sequence, indexed by $k$, showing that~\eqref{e:liminf} is satisfied.

Finally, to see that $(X,\sigma)$ has zero topological entropy note that by~\eqref{e:liminf} there are infinitely many $n$ for which $P_X(n)\leq p_n$.  Combining this fact with the subexponential growth condition on the sequence $\{p_n\}$ given in~\eqref{eq:subexp}, we have that 
$$
\liminf_{n\to\infty}\frac{\log P_X(n)}{n}\leq\limsup_{n\to\infty}\frac{\log p_n}{n}=0.
$$
Since the limit defining topological entropy exists, it follows that 
\begin{equation*}
h_{top}(X)=\lim_{n\to\infty}\frac{\log P_X(n)}{n}=\liminf_{n\to\infty}\frac{\log P_X(n)}{n}=0.\quad\qedhere
\end{equation*}
\end{proof}

\section*{Appendix: Ferenczi's theorem on the rank of systems of linear
complexity}
\label{appendix}
We would like to conclude from our work here that  
a minimal system that has a non-loosely Bernoulli measure
must have $\liminf_{n\to\infty} P_X(n)/n=\infty$.  If we knew that a minimal subshift whose complexity function satisfies $\liminf_{n\to\infty} P_X(n)/n<\infty$ has finite rank, then as discussed previously we would have that all measures supported on such a subshift are loosely Bernoulli and our conclusion would follow.  
While Ferenczi~\cite{Fer} shows that a minimal subshift whose complexity
function satisfies $P_X(n)=O(n)$ has finite rank, essentially the same proof can be used to show that the result holds under the slightly weaker hypothesis $\liminf_{n\to\infty} P_X(n)/n<\infty$.   As we were not able to find the $\liminf$ version of this result in the
literature, for completeness we include the proof which demonstrates how to use Ferenczi's argument to get this result.

First we recall what it means for a system to have finite rank, taking
the definition directly from~\cite[Definition 8]{Fer}:
\begin{definition*}
A subshift $(X,\sigma)$ and invariant measure $\mu$ has {\em rank at
most $r$} if for every partition $P=(A_1,A_2,\dots,A_{|P|})$ and every
$\varepsilon>0$, there exist $r$ many subsets $F_i\subseteq X$, $r$ many
positive integers $h_i$, and a partition
$P^{\prime}=(B_1,B_2,\dots,B_{|P^{\prime}|})$ such that
\begin{enumerate}
\item all sets of the form $T^jF_i$, where $1\leq i\leq r$ and $0\leq
j<h_i$, are pairwise disjoint;
\item
$$
d(P,P^{\prime}):=\min\left\{\sum_{i=1}^{\max\{|P|,|P^{\prime}|\}}\mu(A_i\triangle
B_{\sigma(i)})\colon\sigma\in\mathrm{Sym}(\max\{|P|,|P^{\prime}|\})\right\}<\varepsilon
$$
where $P$ or $P^{\prime}$ have been padded with null sets to give them
the same number of elements;
\item the elements of $P^{\prime}$ can be expressed as unions 
of elements of the partition consisting of all sets of
the form $T^jF_i$, where $1\leq i\leq r$ and $0\leq j<h_i$, as well as
$X\setminus\bigcup_i\bigcup_jT^jF_i$.
\end{enumerate}
The subshift  $(X,\sigma)$ has {\em finite rank} if it has rank at most
$r$ for some $r\in\N$.
\end{definition*}

While the following lemma is slightly stronger than~\cite[Proposition 4]{Fer}, it readily follows from the proof given there.  For completeness we include the argument.
\begin{lemma*}
Let $(X,\sigma)$ be a subshift and let $\mu$ be an invariant measure
supported on $X$.  If
\begin{equation}\label{eq:bound}
\liminf_{n\to\infty}\frac{P_X(n)}{n}<\infty,
\end{equation}
then $(X,\sigma,\mu)$ has finite rank.
\end{lemma*}

\begin{proof}
By~\eqref{eq:bound}, there is an integer $R$ and an increasing sequence
$\{n_j\}_{j=1}^{\infty}$ such that
\begin{equation}\label{eq:bound2}
P_X(n_j+1)-P_X(n_j)\leq R
\end{equation}
for all $j\geq 1$.  By~\eqref{eq:bound2}, there are at most $R$ many
elements of $\mathcal{L}_{n_j+1}(X)$ whose leftmost subword of length
$n_j$ is a word that fails to extend uniquely to its right.  Passing to
a subsequence if necessary, we can assume there exists $r\leq R$ such
that for all $j$ there are exactly $r$ many elements of
$\mathcal{L}_{n_j}(X)$ that appear as the rightmost subword of length
$n_j$ in a word in $\mathcal{L}_{n_j+1}(X)$ whose leftmost subword of
length $n_j$ fails to extend uniquely to its right.  For each $j$
enumerate these words as
$$
F_1^j,F_2^j,\dots,F_r^j\in\mathcal{L}_{n_j}(X).
$$
Now for each $1\leq k\leq r$ let $h_k\in\N$ be the smallest positive
integer for which there exists $1\leq i\leq r$ such that
$[F_i^j]\cap\sigma^{h_k}[F_k^j]\neq\emptyset$.

Note that all sets of the form $\sigma^k[F_i^j]$, where $1\leq i\leq r$
and $0\leq k<h_i$, are pairwise disjoint by definition of $h_i$.   (These sets are
actually a partition of $X$ and there is no need to add the complement
of their union.)  Also
note that the elements of the partition $Q_j$ of $X$ into cylinder sets of length
$n_j$ are unions of elements of the partition given by sets of the form $\sigma^k[F_i^j]$, where $1\leq i\leq r$ and $0\leq k<h_i$, meaning that $Q_j$ is a coarser partition. 
Finally, note that any partition $P$ of $X$ can be
approximated arbitrarily well by partitions coarser than the partitions
into cylinder sets of increasing length.  Therefore for any $P$ and any
$\varepsilon>0$ there exists $j$ such that some partition, $P^{\prime}$,
coarser than $Q_j$ satisfies $d(P,P^{\prime})<\varepsilon/2$ and such that
$$
d(Q,\{\sigma^k[F_i^j]\colon1\leq i\leq r, 0\leq k<h_i\})<\varepsilon/2.
$$
Thus $(X,\sigma,\mu)$ has rank at most $r$, and in particular it has finite rank.
\end{proof}


\begin{thebibliography}{99}
\bibitem{Boz}
{\sc M.~Boshernitzan.}
A unique ergodicity of minimal symbolic flows with linear block growth. 
{\em J. Analyse Math.} {\bf 44} (1984/85), 77--96. 


\bibitem{CK1}
{\sc V.~Cyr and B.~Kra.}
The automorphism group of a shift of linear growth: beyond transitivity. 
{\em Forum Math. Sigma} {\bf 3} (2015), e5, 27 pp. 

\bibitem{CK3}
{\sc V.~Cyr and B.~Kra.}
The automorphism group of a minimal shift of stretched exponential growth. 
{\em J. Mod. Dyn.} {\bf 10} (2016), 483--495.

\bibitem{CK2}
{\sc V.~Cyr and B.~Kra.}
Counting generic measures for a subshift of linear growth. 
{\em J. Eur. Math. Soc. (JEMS)} {\bf  21} (2019), no. 2, 355--380.

\bibitem{CK}
{\sc V.~Cyr and B.~Kra.}
Realizing ergodic properties in zero entropy subshifts.
To appear, {\em  Ergodic Theory Dynam. Systems}

\bibitem{DDMP}
{\sc S. Donoso, F. Durand, A. Maass, and S. Petite}
On automorphism groups of low complexity subshifts.
{\em  Ergodic Theory Dynam. Systems} {\bf 36} (2016) no.1, 64-95.


\bibitem{dye1}
{\sc H.~A.~Dye.}
On groups of measure preserving transformations. I. 
{\em Amer. J. Math.} {\bf 81} (1959), 119--159. 

\bibitem{dye2}
{\sc H.~A.~Dye.}
On groups of measure preserving transformations. II. 
{\em Amer. J. Math.} {\bf 85} (1963), 551--576. 

\bibitem{feldman}
{\sc J.~Feldman.}
New $K$-automorphisms and a problem of Kakutani. 
{\em Israel J. Math.} {\bf 24} (1976), no. 1, 16--38. 

\bibitem{Fer}
{\sc S.~Ferenczi.} 
Rank and symbolic complexity.
{\em  Ergodic Theory Dynam. Systems} {\bf 16} (1996), no. 4, 663-682.

\bibitem{Fer2}
{\sc S.~Ferenczi.}
Measure-theoretic complexity of ergodic systems.
{\em Israel J. Math.} {\bf 100} (1997) 189-207.


\bibitem{GRK}
{\sc F.~Garc\'ia-Ramos and D.~Kwietniak.}
On topological models of zero entropy loosely Bernoulli systems.  
{\tt arXiv:2005.02484}

\bibitem{GK}
{\sc M. Gerber and P. Kunde}.
Loosely Bernoulli odometer-based systems whose corresponding circular systems are not loosely Bernoulli.
To appear, {\em J. Analyse Math.}  {\tt arXiv:1803.01926}

\bibitem{KVW}
{\sc A.~Kanigowski, K.~Vinhage, and D.~Wei.}
Kakutani Equivalence of Unipotent Flows.  
{\tt arXiv:1805.01501}

\bibitem{katok1}
{\sc A.~Katok.} 
Monotone equivalence in ergodic theory. 
{\em Izv. Akad. Nauk SSSR Ser. Mat.} {\bf 41} (1977), no. 1, 104--157. 

\bibitem{katok}
{\sc A.~Katok.} 
Combinatorial constructions in ergodic theory and dynamics.
 University Lecture Series, {\bf 30}.  American Mathematical Society, Providence, RI, 2003. iv+121 pp.
 
 \bibitem{KS}
 {\sc A.~Katok and A.~Stepin.} 
Approximations in ergodic theory. 
{\em Uspehi Mat. Nauk} {\bf 22} (1967), no. 5 (137), 81--106. 

\bibitem{KL}
 {\sc D.~Kwietniak and M.~L\c acka.}
 Feldman-Katok pseudometric and the GIKN construction of nonhyperbolic ergodic measures. 
 {\tt arXiv:1702.01962}
 
 \bibitem{KW}
{\sc A. Kanigowski and D. Wei.}
Product of two Kochergin flows with different exponents is not standard.
{\em Studia Math.} {\bf 244} (1965), 265--283.

\bibitem{MH}
{\sc M.~Morse \& G.~A.~Hedlund.} Symbolic dynamics II. Sturmian trajectories.  {\em Amer. J. Math.} {\bf 62} (1940) 1--42.


\bibitem{ornstein}
{\sc D.~Ornstein.}
Bernoulli shifts with the same entropy are isomorphic. 
{\em Advances in Math.} {\bf 4} (1970), 337--352. 



\bibitem{ORW}
{\sc D.~Ornstein, D.~Rudolph, and B.~Weiss.}
Equivalence of measure preserving transformations. 
{\em Mem. Amer. Math. Soc.} {\bf 37} (1982), no. 262.  

\bibitem{ratner}
{\sc M.~Ratner.}
Horocycle flows are loosely Bernoulli. 
{\em Israel J. Math.} {\bf 31} (1978), no. 2, 122--132. 

\bibitem{Rud}
{\sc D.~Rudolph.} 
Restricted orbit equivalence. 
{\em Mem. Amer. Math. Soc.} {\bf 54} (1985), no. 323, v+150 pp.


\end{thebibliography}
\end{document}